\documentclass[a4paper,titlepage,12pt]{amsart}

\usepackage[utf8]{inputenc}
\usepackage[T1]{fontenc}
\usepackage[english]{babel}
\usepackage{verbatim}
\usepackage[left=3.4cm,right=3.4cm,top=3.5cm,bottom=3.5cm]{geometry}
\usepackage{amsfonts,amssymb,amsthm,nccmath,mathtools}%\usepackage[intlimits]{amsmath}
\usepackage{braket}
\usepackage{enumerate}
\usepackage{hyperref}
\usepackage{cleveref}
\usepackage{bm}
\usepackage{tikz}\usetikzlibrary{cd}
\usepackage{adjustbox}

\newcommand\N{\mathbb{N}}
\newcommand\Z{\mathbb{Z}}

\newcommand\C{\mathbb{C}}

\newcommand\K{\mathbb{K}}

\newcommand\Def[1]{\textbf{{#1}}}

\newcommand\Comment[1]{\textcolor{violet}{MyComment: #1}}
\newcommand\New[1]{\textcolor{red}{#1}}

\DeclareMathOperator{\GL}{GL}
\DeclareMathOperator{\Flag}{\textit{Fl}}
\DeclareMathOperator{\Gr}{Gr}
\DeclareMathOperator{\rep}{rep}

\DeclareMathOperator{\modCat}{mod}

\DeclareMathOperator{\im}{Im}
\DeclareMathOperator{\Hom}{Hom}
\DeclareMathOperator{\Ext}{Ext}

\DeclareMathOperator{\id}{id}

\DeclareMathOperator{\dimvec}{{\bf dim}}
\DeclareMathOperator{\bs}{{BS}}

\theoremstyle{plain}
\newtheorem{theorem}[equation]{Theorem}
\crefname{theorem}{\bf{Theorem}}{\bf{theorems}}
\newtheorem{lemma}[equation]{Lemma}
\crefname{lemma}{\bf{Lemma}}{\bf{lemmas}}
\newtheorem{proposition}[equation]{Proposition}
\crefname{proposition}{\bf{Proposition}}{\bf{propositions}}
\newtheorem{corollary}[equation]{Corollary}
\crefname{corollary}{\bf{Corollary}}{\bf{corollaries}}
\theoremstyle{remark}
\newtheorem{remark}[equation]{Remark}
\crefname{remark}{\textit{Remark}}{\textit{remarks}}
\theoremstyle{definition}
\newtheorem{definition}[equation]{Definition}
\crefname{definition}{\bf{Definition}}{\bf{definitions}}
\newtheorem{example}[equation]{Example}
\crefname{example}{\bf{Example}}{\bf{examples}}

\theoremstyle{plain}
\newtheorem{maintheorem}{Theorem}
\crefname{maintheorem}{\bf{Theorem}}{\bf{theorems}}

\keywords{Schubert varieties, Quiver Grassmannians, Bott-Samelson resolution}
\subjclass{16G20, 14M15, 14N20.}

\title[Quiver Grassmannians for the Bott-Samelson resolution]{Quiver Grassmannians for the Bott-Samelson resolution of type A Schubert varieties}

\numberwithin{equation}{section}

\begin{document}

\author{Giulia Iezzi}
\address{Chair of Algebra and Representation Theory, RWTH Aachen, Pontdriesch 10-14, 52062 Aachen, Germany}
\email{iezzi@art.rwth-aachen.de}

\begin{abstract}
We realise the Bott-Samelson resolutions of type A Schubert varieties as quiver Grassmannians. In order to explicitly describe this isomorphism, we introduce the notion of a \textit{geometrically compatible} decomposition for any permutation in $S_n$. For smooth type A Schubert varieties, we identify a suitable dimension vector such that the corresponding quiver Grassmannian is isomorphic to the Schubert variety. To obtain these isomorphisms, we construct a special quiver with relations and investigate two classes of quiver Grassmannians for this quiver.
\end{abstract}
{\let\newpage\relax\maketitle}

\section{Introduction}
In this paper we construct a special quiver with relations and a rigid representation for this quiver, to then consider the quiver Grassmannian that corresponds to opportune choices of a dimension vector for the quiver. We show how this quiver Grassmannian can realise the Bott-Samelson resolution for Schubert varieties and how, for a different dimension vector, it is isomorphic to a chosen smooth Schubert variety.
We first give a brief introduction about the objects of study in this work and then present our main results.

Schubert varieties first appeared at the end of the 19$^{\text{th}}$ century in the context of Schubert calculus, whose purpose is to determine the number of solutions of certain intersection problems, and have become some of the best understood examples of complex projective varieties. They have recently been linked to degenerate flag varieties and to quiver Grassmannians. Two examples of such connections are in \cite{irelli2014degenerate}, where the authors show that any type A or C degenerate flag variety is isomorphic to a Schubert variety in an appropriate partial flag manifold, and later in \cite{cerulli2017schubert}, which proves that some Schubert varieties arise as irreducible components of certain quiver Grassmannians.
The Bott-Samelson(-Demazure-Hansen) varieties provide natural resolutions of Schubert varieties. They were introduced independently by Demazure and Hansen, and named Bott-Samelson by Demazure \cite{demazure1974desingularisation}.

Given a quiver $Q$ and a $Q$-representation $M$, the quiver Grassmannian $\Gr_{\bf e}(M)$ is the projective variety of subrepresentations $N\subseteq M$ of dimension vector $\mathbf{e}$. They first appeared in \cite{crawle-boevey1989quiver, schofield1992quiver} and have since been extensively studied, for instance as a tool in cluster algebra theory \cite{caldero2006cluster} or for studying linear degenerations of the flag variety \cite{fourier2020lineardegenerations,cerulliirelli2012quiveranddegenerate,feigin2010grassmanndegenerations, feigin2013frobeniussplittin}.
Notably, every projective variety arises as a quiver Grassmannian \cite{reineke2013projectiveisquiver}, or, more generally, as the quiver Grassmannian of every wild quiver \cite{ringel2018quiver}.

In this work, we first define a special class of quivers with relations $(\Gamma, I)$ and a $(\Gamma, I)$-representation $M$ (see Definition \eqref{def:repM}). Given, for instance, an ambient dimension $n+1=4$, then $(\Gamma, I)$ and $M$ are
\begin{equation*}
\begin{tikzcd}[row sep=small]
\overset{\C}{\bullet} \ar[r, "\id"] \ar[d, "\iota_{2,1}"]  & \overset{\C}{\bullet} \ar[r, "\id"] \ar[d, "\iota_{2,1}"] & \overset{\C}{\bullet} \ar[d, "\iota_{2,1}"] \\
\overset{\C^2}{\bullet} \ar[ur, phantom, "\scalebox{1.5}{$\circlearrowleft$}"] \ar[r, "\id"] \ar[d, "\iota_{3,2}"] & \overset{\C^2}{\bullet} \ar[ur, phantom, "\scalebox{1.5}{$\circlearrowleft$}"] \ar[r, "\id"] \ar[d, "\iota_{3,2}"]  & \overset{\C^2}{\bullet}  \ar[d, "\iota_{3,2}"]\\
\overset{\C^{3}}{\bullet} \ar[ur, phantom, "\scalebox{1.5}{$\circlearrowleft$}"] \ar[r, "\id"] \ar[d, "\iota_{4,3}"] & \overset{\C^{3}}{\bullet} \ar[ur, phantom, "\scalebox{1.5}{$\circlearrowleft$}"] \ar[r, "\id"] \ar[d, "\iota_{4,3}"] & \overset{\C^{3}}{\bullet} \ar[d, "\iota_{4,3}"] \\
\overset{\C^{4}}{\bullet} \ar[ur, phantom, "\scalebox{1.5}{$\circlearrowleft$}"] \ar[r, "\id"] & \overset{\C^{4}}{\bullet} \ar[ur, phantom, "\scalebox{1.5}{$\circlearrowleft$}"] \ar[r, "\id"] & \overset{\C^{4}}{\bullet} 
\end{tikzcd}.
\end{equation*}
Because of the shape of this particular quiver, it is convenient to visualise it as a grid, or matrix, and denote its vertices by double indices $(i,j)$. Consequently, given $M$ as above (and analogously for any other $(\Gamma, I)$-representation), we denote by $M_{i,j}$ the vector space associated to vertex $(i,j)$.

Our first result is the following:

\begin{maintheorem}\label{thm:1}
    $M$ is a rigid representation of $(\Gamma, I)$.
\end{maintheorem}

Applying \cite[Proposition 3.8]{irelli2021cell} and \cite[Proposition 2.2]{cerulliirelli2012quiveranddegenerate}, we deduce from Theorem \ref{thm:1} that any quiver Grassmannian associated to the $(\Gamma, I)$-representation $M$ is a smooth and irreducible variety.

We then consider a Schubert variety $X_w$ in the flag variety $\Flag_{n+1}$, for any fixed permutation $w\in S_{n+1}$, and define the dimension vector ${\bf r}^w$ for $(\Gamma, I)$ according to $w$.

\begin{maintheorem}\label{thm:2}
  The quiver Grassmannian $\Gr_{{\bf r}^w}(M)$ is isomorphic to any Bott-Samelson resolution of $X_w$ associated to a geometrically compatible decomposition of $w$.
\end{maintheorem}

As part of the proof of Theorem \ref{thm:2} (see Theorem \ref{thm:BSisom}) we  give an explicit description of the isomorphism between such a Bott-Samelson resolution of $X_w$ and our quiver Grassmannian $\Gr_{{\bf r}^w}(M)$. We remark that the quiver, the representation and the corresponding quiver Grassmannian we use are different from the ones that would be obtained from, for instance, the general construction for quiver Grassmannians given in \cite{reineke2013projectiveisquiver}. This allows us not only to exploit an algebraic property of the representation $M$ (i.e. its rigidity) to deduce geometrical properties of the associated variety $\Gr_{{\bf r}^w}(M)$, but also to give a straightforward correspondence between the points of $\Gr_{{\bf r}^w}(M)$ and those of the Bott-Samelson resolution of $X_w$ which relies only on the combinatorial definition of the latter (see Definition \ref{def:BS}).

Finally, we make use of a combinatorial characterisation of smooth Schubert varieties: it was first proved in \cite{lakshmibai1990criterion} that a Schubert variety $X_w$ is smooth if and only if $w$ avoids the patterns $[4231]$ and $[3412]$, and this criterion was later characterised in \cite{gasharov2002cohomology} in terms of the conditions that define the flags in $X_w$. We exploit this characterisation and provide an explicit isomorphism between a fixed smooth Schubert variety and the quiver Grassmannian $\Gr_{{\bf e}^w}(M)$ of $(\Gamma, I)$, for a special dimension vector ${\bf e}^w$:

\begin{maintheorem}
If $w\in S_{n+1}$ avoids the patterns $[4231]$ and $[3412]$, the quiver Grassmannian $\Gr_{{\bf e}^w}(M)$ is isomorphic to the Schubert variety $X_w$ and the isomorphism is given by
\begin{gather*}
\begin{aligned}
\varphi: \Gr_{\bf{e}}(M) &\to X_w \\
N &\mapsto N\bm{.}
\end{aligned}
\end{gather*}
where $N\bm{.}=N_{n+1,1}\subseteq N_{n+1,2}\subseteq\dots\subseteq N_{n+1,n}\,$.
\end{maintheorem}

The paper is organised as follows: Section 2 and 3 are dedicated, respectively, to basic facts about quiver Grassmannians and about Schubert varieties in the flag variety; in Section 4 we define the quiver $(\Gamma, I)$, its representation $M$ and prove Theorem 1; Section 5 is mainly concerned with proving that all permutations admit a certain reduced decomposition, called $\textit{geometrically compatible}$ decomposition, in order to prove Theorem 2; in Section 6 we define a special dimension vector for $(\Gamma, I)$ and prove Theorem 3.

\section*{Acknowledgments}
The author would like to thank Giovanni Cerulli Irelli, Evgeny Feigin, Alexander Pütz and Markus Reineke for helpful discussions. Many thanks to Ghislain Fourier and Martina Lanini for the insightful discussions and comments on this project. Further thanks to David Anderson and Markus Reineke for pointing out a few corrections in the first version of this paper.
This work is a contribution to the SFB-TRR 195 “Symbolic Tools in Mathematics and their Applications” of the German Research Foundation (DFG).

\section{Background on quiver Grassmannians}

We first collect some facts about quiver representations and quiver Grassmannians. Standard references are \cite{crawley1992lectures,schiffler2014quiver}.

\begin{definition}
    A finite \Def{quiver} $Q=(Q_0,Q_1,s,t)$ is given by a finite set of vertices $Q_0$, a finite set of arrows $Q_1$ and two maps $s,t: Q_1\to Q_0$ assigning to each arrow its source, resp. target.
 \end{definition}

 \begin{definition}
     Given a quiver $Q$, the finite-dimensional $Q$-\Def{representation} $M$ over an algebraically closed field $\K$ is the ordered pair $((M_i)_{i\in Q_0},(M^{\alpha})_{\alpha\in Q_1})$, where $M_i$ is a finite-dimensional $\K$-vector space attached to vertex $i\in Q_0$ and $M^{\alpha}:M_{s(\alpha)}\to M_{t(\alpha)}$ is a $\K$-linear map for any $\alpha \in Q_1$.
     The \Def{dimension vector} of $M$ is $\dimvec M \coloneqq(\dim_{\K}M_i)_{i\in Q_0} \in \Z^{|Q_0|}_{\geq 0}$.

     A \Def{subrepresentation} of $M$, denoted by $N=((N_i)_{i\in Q_0},(M^{\alpha}\restriction_{N_{s(\alpha)}})_{\alpha\in Q_1})$, is a $Q$-representation such that $N_i\subseteq M_i $ for any $ i\in Q_0$ and $M^{\alpha}(N_{s(\alpha)})\subseteq N_{t(\alpha)}$ for any $\alpha\in Q_1$.
     \end{definition}

 The finite-dimensional $Q$-representations over $\K$ form a category, denoted by $\rep_{\K}(Q)$, where a morphism $\phi$ between $M$ and $M'$ in $\rep_{\K}(Q)$ is given by linear maps $\phi_i:M_i\to M'_i \quad \forall i \in Q_0$ such that $\phi_{t(\alpha)} \circ M^{\alpha}=M'^{\alpha}\circ \phi_{s(\alpha)}$. It is known (see \cite[Theorem 5.4]{schiffler2014quiver} for a proof) that $\rep_{\K}(Q)$ is equivalent to the category $A$-$\modCat$ of finite-dimensional modules over the path algebra $A=\K Q$ of $Q$, and that it is abelian, Krull-Schmidt and hereditary. For $M,N \in \rep_{\K}(Q)$ we use the standard notation
\begin{equation*}
    [M,N]\coloneqq \dim_{\K}\Hom_Q(M,N), \; [M,M]^1\coloneqq \dim_{\K}\Ext^1_Q(M,M).
\end{equation*}

We call a representation $M$ \Def{rigid} if it has no self-extensions, which means $[M,M]^1=0$.
Finally, we denote by
\begin{equation*}
\langle M,N \rangle = [M,N]-[M,N]^1
\end{equation*}
the Euler-Ringel form of $Q$, which can be computed (see for instance \cite{bongartz1983algebras} for details) via the bilinear form $\langle -,-\rangle : \Z^{|Q_0|} \times \Z^{|Q_0|} \to \Z$ defined as
\begin{equation}\label{eq:eulerform}
   \langle \textbf{d} ^M, \textbf{d} ^N \rangle = \langle \dimvec M, \dimvec N \rangle \coloneqq \sum_{i \in Q_0} d_i^M d_i^N - \sum_{\alpha \in Q_1} d_{s(\alpha)}^M d_{t(\alpha)}^N.
\end{equation}

\begin{definition}
    A \Def{relation} on a quiver $Q$ is a subspace of the path algebra of $Q$ spanned by linear combinations of paths with common source and target, of length at least 2.
    Given a two-sided ideal $I$ of $\K Q$ generated by relations, the pair $(Q,I)$ is a \Def{quiver with relations} and the quotient algebra $\K Q/I$ is the path algebra of $(Q,I)$.
\end{definition}

A system of relations for $I$ is defined as a subset $R$ of $\cup_{i,j\in Q_0} iIj$, where $i$ denotes the trivial path on vertex $i$, such that $R$, but no proper subset of $R$, generates $I$ as a two-sided ideal. For any two vertices $i$ and $j$, we denote by $r(i,j,R)$ the cardinality of the set $R\cap iIj$, which contains those elements in $R$ that are linear combinations of paths starting in $i$ and ending in $j$. If $Q$ contains no oriented cycle, then the numbers $r(i,j,R)$ are independent of the chosen system of relations (see for instance \cite{bongartz1983algebras}), and can therefore be denoted by $r(i,j)$.
The category $\rep_{\K}(Q,I)$ consists of the finite-dimensional representations of $Q$ that satisfy the relations in $I$.

For a quiver with relations $(Q,I)$ and no oriented cycles, the Euler-Ringel form is given by
\begin{equation}\label{eq:eulerformrelations}
   \langle \textbf{d} ^M, \textbf{d}^N \rangle = \sum_{i \in Q_0} d_i^M d_i^N - \sum_{\alpha \in Q_1} d_{s(\alpha)}^M d_{t(\alpha)}^N + \sum_{i,j \in Q_0} r(i,j) d_i^M d_j^N.
\end{equation}
More background and details can be found for instance in \cite[Section 2.2]{bongartz1983algebras} or \cite{derksen2002semi}. To simplify notation, we will sometimes denote a $Q$-representation $M$ by its tuple of vector spaces $(M_i)_{i \in Q_0}$ when the assignment of the linear maps is clear from context.

\begin{definition}
    Consider a quiver $Q$, a $Q$-representation $M$ and a dimension vector $\bf e\in\Z^{Q_0}_{\geq 0}$ such that $e_i\leq \dim M_i \; \forall i\in Q_0$. The \Def{quiver Grassmannian} $\Gr_{\bf e}(M)$ is defined as the collection of all subrepresentations $N$ of $M$ with $\dim N_i=e_i $ for all $ i\in Q_0$.
\end{definition}

    Analogous to Grassmannians and flag varieties, non-empty quiver Grassmannians can be realised as closed subvarieties of products of Grassmannians, via the close embedding
    \begin{equation*}
        \iota: \Gr_{\bf e}(M) \to \prod_{i \in Q_0}\Gr(e_i,M_i)
    \end{equation*}
    which sends a subrepresentation $N$ of $M$ to the collection of $e_i$-dimensional subspaces $N_i$ of $M_i$.
    %The relations defining (pointwise) the subvariety associated to a given quiver representation and a dimension vector are given in \cite{lorscheid2019pluckerelations}.

\begin{example}[The flag variety]\label{ex: flag variety}

    In \cite[Proposition 2.7]{cerulliirelli2012quiveranddegenerate}, the authors realise the (linear degenerate) flag variety as the quiver Grassmannian associated to certain representations of the equioriented quiver of type $\mathbb{A}_n$. In particular, the complete flag variety $\Flag_{n+1}$ in $\C^{n+1}$ can be realised as follows.

    Consider the quiver with $n$ vertices, ordered from 1 to $n$, and $n-1$ arrows of the form $i\to i+1$. We fix the  dimension vector ${\bf e}=(1,2,\dots,n)$ and the representation $M$ with $M_i=\C^{n+1}$ for $i=1,\dots,n$ and $M^{\alpha}=\id$ for all arrows $\alpha$:
\begin{equation*}
\begin{tikzcd}[]
\overset{\C^{n+1}}{\bullet} \ar[r, "\id"] & \overset{\C^{n+1}}{\bullet} \ar[r, "\id"] & ... \ar[r, "\id"] & \overset{\C^{n+1}}{\bullet}.\\
\end{tikzcd}
\vspace{-0.7cm}
\end{equation*}
 The quiver Grassmannian $\Gr_{\bf e}(M)$ consists precisely of the subrepresentations $N$ of $M$ with $\dim(N_i)=i$, i.e. full flags of vector subspaces.
\end{example}

\section{Background on Schubert varieties}\label{sec:schubvar}
Given any $v_1,\dots,v_r$ in $\C^{n+1}$, we denote by $\langle v_1,\dots, v_r \rangle$ their $\C$-linear span. To define Schubert varieties in $\Flag_{n+1}$, we first fix a basis $\mathcal{B}=\set{b_1, b_2,\dots, b_{n+1}}$ of $\C^{n+1}$ and denote by $F\bm{.}$ the standard flag $\langle b_1\rangle \subseteq  \langle b_1, b_2\rangle \subseteq \dots \subseteq \langle b_1, b_2,\dots, b_{n+1} \rangle$ and by $S_{n+1}$ the symmetric group on $n+1$ elements. More facts and details about Schubert varieties can be found for instance in \cite[Part III]{fulton1997young}.

\begin{definition}\label{def:schubcell}
For $w \in S_{n+1}$, the \Def{Schubert cell} $X_w^\circ$ is
\begin{equation*}
X^\circ_w = \set{V\bm{.}\in \Flag_{n+1}: \dim(F_p\cap V_q)= \#\set{k\leq q : w(k)\leq p},\; 1 \leq p, q \leq n+1}.
\end{equation*}
\end{definition}

\begin{definition}\label{def:schubvar}
The \Def{Schubert variety} $X_w$ is defined as the closure in $\Flag_{n+1}$ of the cell $X_w^\circ$, that is
\begin{equation*}
X_w = \set{V\bm{.}\in \Flag_{n+1}: \dim(F_p\cap V_q)\geq \#\set{k\leq q : w(k)\leq p}, \;1 \leq p, q \leq n+1}.
\end{equation*}
\end{definition}

We observe that the conditions on the intersections between the $F_p$ and the $V_q$ imply, for each pair $p,q$, one of the following : $F_p\subset V_q$, $F_p\supset V_q$, $F_i=V_j$ or $F_p\cap V_q= U$ with $0\leq\dim(U)<\min\set{p,q}$. A minimal set of conditions that imply all the conditions defining a Schubert variety $X_w$ has been described in terms of essential sets of the permutation $w$ (see \cite[Section 3]{fulton1992flags} or \cite[Section 4]{gasharov2002cohomology}).
Each Schubert variety $X_w$ is an irreducible subvariety of $\Flag_{n+1}$, and its dimension is given by the number of inversions in $w$, called length: $$\ell(w)= \#\{ i<j : w(i)>w(j) \}.$$ 
The length of a permutation $w$ is also the minimal number of simple transpositions needed to form a decomposition of $w$, called reduced decomposition: $w=s_{\ell(w)}\cdots s_1$, where $s_i$ denotes the swap of $i$ and $i+1$. We recall that, in general, a permutation admits more than one reduced decomposition.

The Schubert variety $X_w$ consists of the cell $X_w^\circ$, which is open and dense in $X_w$, and of the cells corresponding to permutations that are smaller than $w$ with respect to the Chevalley-Bruhat order $\leq$ on $S_{n+1}$: $ X_w = \mathop{\sqcup}_{u\leq w} X_u^{\circ}$.
The relation $u\leq w$ holds under the Chevalley-Bruhat (partial) order on $S_{n+1}$ if a reduced expression of $w$ contains a subexpression which is a reduced expression for $u$. 

For the purposes of this paper, we consider the intersections $F_p\cap V_q$ instead of the standard $V_p\cap F_q$ in the definition of Schubert varieties, and to simplify the notation we write $r^w_{p,q}$ for the numbers $\#\set{k\leq q : w(k)\leq p}$.
We represent a permutation $w$ in $S_n$ by listing its (naturally) ordered images, that is, its one-line notation $w=[w(1)w(2)\dots w(n)]$.

%This is because we want to build a correspondance between the $d_{p,q}$ and the entries of a matrix, where the row index $p$ corresponds to the subspace $V_p$ and $q$ to $F_q$.

\begin{example}
For $e=[1\,2\dots n+1]$ and $w_0= [n+1 \, n\dots 1]$ in $S_{n+1}$, it is easy to compute from Definition \ref{def:schubvar} the Schubert varieties of minimal and maximal dimension, respectively $X_e = \{ F\bm{.} \}$ and $X_{w_0}=\Flag_{n+1}$.
\begin{comment}
For $n+1=3$ and $w=123 \in S_3$, we compute from \ref{def:schubvar} the conditions defining $X_{123}$:
\begin{center}
\begin{tabular}{c c c}
     $d_{1,1}\geq 1$  & $d_{1,2}\geq 1$ & $d_{1,3}\geq 1$ \\
     $d_{2,1}\geq 1$  & $d_{2,2}\geq 2$ & $d_{2,3}\geq 2$ \\
     $d_{3,1}\geq 1$  & $d_{3,2}\geq 2$ & $d_{3,3}\geq 3$
\end{tabular}
\end{center}
which imply that $X_{123}$ consists of one point, the standard flag. For $w=321$ we obtain
\begin{center}
\begin{tabular}{c c c}
     $d_{1,1}\geq 0$  & $d_{1,2}\geq 0$ & $d_{1,3}\geq 1$ \\
     $d_{2,1}\geq 0$  & $d_{2,2}\geq 1$ & $d_{2,3}\geq 2$ \\
     $d_{3,1}\geq 1$  & $d_{3,2}\geq 1$ & $d_{3,3}\geq 3$
\end{tabular}.
\end{center}
These conditions are satisfied by every flag in $\Flag_3$, therefore $X_{321}=\Flag_3$. In general,
\end{comment}
\end{example}

%\subsection{Schubert varieties and smoothness}

Smooth Schubert varieties were characterised combinatorially in \cite{lakshmibai1990criterion}: a Schubert variety $X_w$ is smooth if and only if $w$ avoids the patterns $[4231]$ and $[3412]$. We recall that a permutation $w=[w(1)w(2)\dots w(n)]$ avoids a pattern $\pi$ if no subsequence of $w$ has the same relative order as the entries of $\pi$.
In \cite[Theorem 1.1]{gasharov2002cohomology}, the authors prove that this pattern-avoiding condition is equivalent to $X_w$ being defined by non-crossing inclusions:

\begin{definition}(\cite[Section 1]{gasharov2002cohomology})\label{def:inclusions}
A Schubert variety $X_w$ is \Def{defined by inclusions} if the defining conditions on each $V_q$ (see Definition \ref{def:schubvar}) are a conjunction of conditions of the form $V_q \subseteq F_p$ and $V_q \supseteq F_{s}$, for some $p$ and $s$. A pair of conditions $V_{q}\subset F_{p}$ and $F_{p'}\subset V_{q'} $ is \Def{crossing} if $q<q'$ and $p>p'$.

If $X_w$ is defined by inclusions and its conditions do not contain any crossing pair, then $X_w$ is \Def{defined by non-crossing inclusions}.
\end{definition}

\begin{example}
All permutations in $S_3$ are defined by non-crossing inclusions.

In $S_5$, the permutation $w=[31542]$ avoids both patterns $[4231]$ and $[3412]$, which means that $X_{w}$ is defined by non-crossing inclusions. We can compute these inclusions using Definition \ref{def:schubvar}: a flag $V\bm{.}$ is in $X_{w}$ if and only if
\begin{equation*}
    V_1 \subseteq F_3,\; F_1 \subseteq V_2\subseteq F_3,\; F_1\subseteq V_3,\; F_1\subseteq V_4 .
\end{equation*}
The same conditions can be described without redundancy as $F_1 \subseteq V_2\subseteq F_3$, which is a pair of non-crossing inclusions.

A permutation in $S_5$ that yields crossing inclusions is, for instance, \mbox{$\tau=[45312]$}, which contains the pattern $[3412]$. A flag $V\bm{.}$ is in $X_{\tau}$ if and only if $V_1\subseteq F_4$ and $ F_1\subseteq V_4$.
Finally, the permutation $\pi=[53421]$ in $S_5$ contains the pattern $[4231]$ and defines a non-trivial condition on $X_{\pi}$ that is not an inclusion: a flag $V\bm{.}$ is in $X_{\pi}$ if and only if $\dim(F_3\cap V_2)\geq 1$.
\end{example}

\section{A special quiver with relations and a special representation}\label{sec:quiver}

We now define a quiver with relations and construct a rigid representation for this quiver. This will then be exploited in sections \ref{sec:desing} and \ref{sec:smooth}, together with two different, appropriate choices of dimension vectors, to recover the Bott-Samelson resolution for Schubert varieties and to realise smooth Schubert varieties as the corresponding quiver Grassmannian, respectively.

The following construction of the quiver and its special representation depends on the fixed ambient dimension $n+1$ but not on the chosen Schubert variety (that is, it's independent of the specific permutation $w$ in $S_{n+1}$).

Given $n\in \N_{\geq 2}$, we consider the following
quiver $\Gamma=(\Gamma_0,\Gamma_1):$
\begin{comment}
\begin{equation*}
\begin{tikzcd}[]
\overset{}{\bullet} \ar[r, ""] \ar[d, ""]  & \overset{}{\bullet} \ar[r, ""] \ar[d, ""] & ...\ar[r, ""] \ar[d, ""] & \overset{}{\bullet} \ar[d, ""] \\
\overset{}{\bullet} \ar[ur, phantom, "\scalebox{1.5}{$\circlearrowleft$}"] \ar[r, ""] \ar[d, ""] & \overset{}{\bullet} \ar[ur, phantom, "\scalebox{1.5}{$\circlearrowleft$}"] \ar[r, ""] \ar[d, ""] & ... \ar[r, ""] \ar[d, ""] & \overset{}{\bullet} \ar[d, ""] \\
... \ar[ur, phantom, "\scalebox{1.5}{$\circlearrowleft$}"] \ar[r, ""] \ar[d, ""] & ... \ar[r, ""] \ar[d, ""] & ... \ar[r, ""] \ar[d, ""] & ... \ar[d, ""]\\
\overset{}{\bullet} \ar[r, ""] & \overset{}{\bullet} \ar[r, ""] & ... \ar[r, ""] & \overset{}{\bullet}\\
\end{tikzcd}.
\end{equation*}
\end{comment}

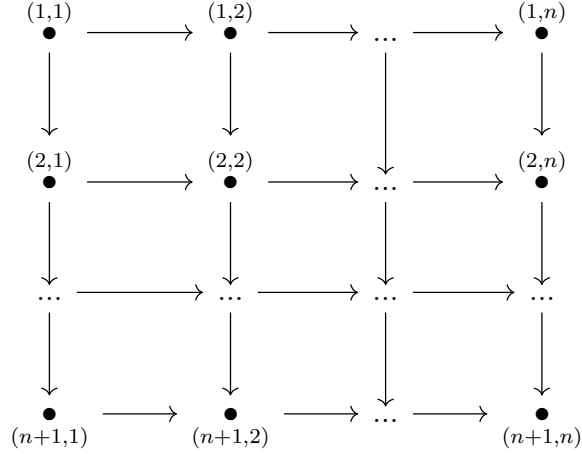
\begin{equation*}
\begin{tikzcd}[row sep= large]
\overset{(1,1)}{\bullet} \ar[r, ""] \ar[d, ""]  & \overset{(1,2)}{\bullet} \ar[r, ""] \ar[d, ""] & ...\ar[r, ""] \ar[d, ""] & \overset{(1,n)}{\bullet} \ar[d, ""] \\
\overset{(2,1)}{\bullet}  \ar[r, ""] \ar[d, ""] & \overset{(2,2)}{\bullet}  \ar[r, ""] \ar[d, ""] & ... \ar[r, ""] \ar[d, ""] & \overset{(2,n)}{\bullet} \ar[d, ""] \\
... \ar[r, ""] \ar[d, ""] & ... \ar[r, ""] \ar[d, ""] & ... \ar[r, ""] \ar[d, ""] & ... \ar[d, ""]\\
\underset{(n+1,1)}{\bullet} \ar[r, ""] & \underset{(n+1,2)}{\bullet} \ar[r, ""] & ... \ar[r, ""] & \underset{(n+1,n)}{\bullet}
\end{tikzcd}
\end{equation*}

\begin{comment}
\begin{equation*}
\begin{tikzcd}[]
\overset{(1,1)}{\bullet} \ar[r, ""] \ar[d, ""]  & \overset{(1,2)}{\bullet} \ar[r, ""] \ar[d, ""] & ...\ar[r, ""] \ar[d, ""] & \overset{(1,n)}{\bullet} \ar[d, ""] \\
\overset{(2,1)}{\bullet} \ar[ur, phantom, "\scalebox{1.5}{$\circlearrowleft$}"] \ar[r, ""] \ar[d, ""] & \overset{(2,2)}{\bullet} \ar[ur, phantom, "\scalebox{1.5}{$\circlearrowleft$}"] \ar[r, ""] \ar[d, ""] & ... \ar[ur, phantom, "\scalebox{1.5}{$\circlearrowleft$}"] \ar[r, ""] \ar[d, ""] & \overset{(2,n)}{\bullet} \ar[d, ""] \\
... \ar[ur, phantom, "\scalebox{1.5}{$\circlearrowleft$}"] \ar[r, ""] \ar[d, ""] & ... \ar[ur, phantom, "\scalebox{1.5}{$\circlearrowleft$}"] \ar[r, ""] \ar[d, ""] & ... \ar[ur, phantom, "\scalebox{1.5}{$\circlearrowleft$}"] \ar[r, ""] \ar[d, ""] & ... \ar[d, ""]\\
\underset{(n+1,1)}{\bullet} \ar[ur, phantom, "\scalebox{1.5}{$\circlearrowleft$}"] \ar[r, ""] & \underset{(n+1,2)}{\bullet} \ar[ur, phantom, "\scalebox{1.5}{$\circlearrowleft$}"] \ar[r, ""] & ... \ar[ur, phantom, "\scalebox{1.5}{$\circlearrowleft$}"] \ar[r, ""] & \underset{(n+1,n)}{\bullet}
\end{tikzcd}
\end{equation*}
\end{comment}

where each vertex in $\Gamma_0$ is labelled by a pair $(i,j)$, for $i=1,\dots,n+1 $ and $j=1,\dots,n$. We denote by $\alpha_{(i,j)}^{(k,l)}$ the arrow going from vertex $(i,j)$ to vertex $(k,l)$.

Then, we consider the following relations on $\Gamma$:
\begin{equation}
    \alpha_{(i,j+1)}^{(i+1,j+1)} \alpha_{(i,j)}^{(i,j+1)}=\alpha_{(i+1,j)}^{(i+1,j+1)} \alpha_{(i,j)}^{(i+1,j)}
\end{equation}
for $i=1,\dots,n$, $j=1,\dots,n-1$, and denote by $I$ the ideal of $\C \Gamma$ generated by these relations. We write $(\Gamma, I)$ for the quiver with relations:

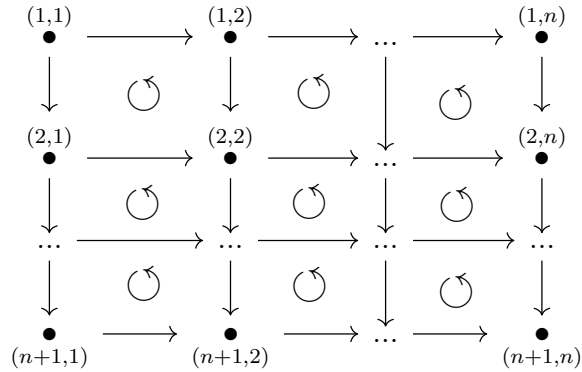
\begin{equation*}
\begin{tikzcd}[]
\overset{(1,1)}{\bullet} \ar[r, ""] \ar[d, ""]  & \overset{(1,2)}{\bullet} \ar[r, ""] \ar[d, ""] & ...\ar[r, ""] \ar[d, ""] & \overset{(1,n)}{\bullet} \ar[d, ""] \\
\overset{(2,1)}{\bullet} \ar[ur, phantom, "\scalebox{1.5}{$\circlearrowleft$}"] \ar[r, ""] \ar[d, ""] & \overset{(2,2)}{\bullet} \ar[ur, phantom, "\scalebox{1.5}{$\circlearrowleft$}"] \ar[r, ""] \ar[d, ""] & ... \ar[ur, phantom, "\scalebox{1.5}{$\circlearrowleft$}"] \ar[r, ""] \ar[d, ""] & \overset{(2,n)}{\bullet} \ar[d, ""] \\
... \ar[ur, phantom, xshift=5, "\scalebox{1.5}{$\circlearrowleft$}"] \ar[r, ""] \ar[d, ""] & ... \ar[ur, phantom,xshift=3, "\scalebox{1.5}{$\circlearrowleft$}"] \ar[r, ""] \ar[d, ""] & ... \ar[ur, phantom,xshift=2, "\scalebox{1.5}{$\circlearrowleft$}"] \ar[r, ""] \ar[d, ""] & ... \ar[d, ""]\\
\underset{(n+1,1)}{\bullet} \ar[ur, phantom, "\scalebox{1.5}{$\circlearrowleft$}"] \ar[r, ""] & \underset{(n+1,2)}{\bullet} \ar[ur, phantom, "\scalebox{1.5}{$\circlearrowleft$}"] \ar[r, ""] & ... \ar[ur, phantom, "\scalebox{1.5}{$\circlearrowleft$}"] \ar[r, ""] & \underset{(n+1,n)}{\bullet}
\end{tikzcd}
\end{equation*}

Now, we define the $(\Gamma, I)$-representation $M=((M_{i,j})_{(i,j)\in \Gamma_0},(M^{\alpha})_{\alpha \in \Gamma_1})$ as 
\begin{equation}\label{def:repM}
M_{i,j}=\C^i, \quad
M^{\alpha}=\begin{cases}
    \iota_{i+1,i} & \text{if } s(\alpha)=(i,j), t(\alpha)=(i+1,j) \\
    \id & \text{if } s(\alpha)=(i,j), t(\alpha)=(i,j+1)
\end{cases}
\end{equation}
where $\iota_{i+1,i}$ denotes the inclusion of $\C^i$ into $\C^{i+1}$, represented with respect to the chosen basis $\mathcal{B}=\set{b_1, b_2,\dots, b_{n+1}}$ by the matrix $\iota_{i+1,i}=\begin{bsmallmatrix}
1&0&\dots&0\\
0&1&\dots&0\\
\dots&\dots&\dots&\dots\\
0&0&0&1\\
0&0&0&0
\end{bsmallmatrix}$.

%We write the dimension vector of $M$ as ${\bf d} =(d_{i,j})$, with $d_{i,j}= i$, for $i=1,\dots,n+1$ and $j=1,\dots,n$.

The relations imposed on $\Gamma$ are trivially satisfied by the representation $M$:
\begin{equation*}
\begin{tikzcd}[sep=large]
\overset{\C}{\bullet} \ar[r, "\id"] \ar[d, "\iota_{2,1}"]  & \overset{\C}{\bullet} \ar[r, "\id"] \ar[d, "\iota_{2,1}"] & ...\ar[r, "\id"] \ar[d, "\iota_{2,1}"] & \overset{\C}{\bullet} \ar[d, "\iota_{2,1}"] \\
\overset{\C^2}{\bullet} \ar[ur, phantom, "\scalebox{1.5}{$\circlearrowleft$}"] \ar[r, "\id"] \ar[d, "\iota_{3,2}"] & \overset{\C^2}{\bullet} \ar[ur, phantom, "\scalebox{1.5}{$\circlearrowleft$}"] \ar[r, "\id"] \ar[d, "\iota_{3,2}"]  & ... \ar[ur, phantom,xshift=5, "\scalebox{1.5}{$\circlearrowleft$}"] \ar[r, "\id"] \ar[d, "\iota_{3,2}"] & \overset{\C^2}{\bullet} \ar[d, "\iota_{3,2}"] \\
... \ar[ur, phantom,xshift=5, "\scalebox{1.5}{$\circlearrowleft$}"] \ar[r, "\id"] \ar[d, "\iota_{n+1,n}"] & ... \ar[ur, phantom,xshift=5, "\scalebox{1.5}{$\circlearrowleft$}"] \ar[r, "\id"] \ar[d, "\iota_{n+1,n}"] & ... \ar[ur, phantom, xshift=5, "\scalebox{1.5}{$\circlearrowleft$}"] \ar[r, "\id"] \ar[d, "\iota_{n+1,n}"] & ... \ar[d, "\iota_{n+1,n}"]\\
\overset{\C^{n+1}}{\bullet} \ar[ur, phantom, "\scalebox{1.5}{$\circlearrowleft$}"] \ar[r, "\id"] & \overset{\C^{n+1}}{\bullet} \ar[ur, phantom, "\scalebox{1.5}{$\circlearrowleft$}"] \ar[r, "\id"] & ... \ar[ur, phantom, xshift=5, "\scalebox{1.5}{$\circlearrowleft$}"] \ar[r, "\id"] & \overset{\C^{n+1}}{\bullet}\\
\end{tikzcd}.
\end{equation*}

We now want to show that $M$ is a rigid representation of $(\Gamma, I)$. To do so, we first consider the following subquiver $\Gamma'$ of $\Gamma$
\begin{equation*}
\begin{tikzcd}[]
\overset{(1,1)}{\bullet} \ar[r, "\iota_{2,1}"] & \overset{(2,1)}{\bullet} \ar[r, "\iota_{3,2}"] & \overset{(3,1)}{\bullet}  \ar[r, "\iota_{4,3}"] & \cdots \ar[r, "\iota_{n+1,n}"] & \overset{(n+1,1)}{\bullet}\\
\end{tikzcd},
\end{equation*}
i.e. the equioriented Dynkin quiver $\mathbb{A}_{n+1}$, and call $M'$ the restriction of the representation $M$ to $\Gamma'$:
\begin{equation*}
\begin{tikzcd}[]
\overset{\C}{\bullet} \ar[r, "\iota_{2,1}"] & \overset{\C^2}{\bullet} \ar[r, "\iota_{3,2}"] & \overset{\C^3}{\bullet}  \ar[r, "\iota_{4,3}"] & \cdots \ar[r, "\iota_{n+1,n}"] & \overset{\C^{n+1}}{\bullet}\\
\end{tikzcd}.
\end{equation*}

\begin{lemma}\label{lemma:subquiver}
$M'$ is a rigid representation of $\Gamma'$.
\end{lemma}
\begin{proof}
We denote by $U_{i,j}$ the indecomposable representation of $\mathbb{A}_{n+1}$ supported on the vertices $i,\dots,j$, for $1\leq i\leq j\leq n+1$, that is:
\begin{equation*}
\begin{tikzcd}[]
\overset{0}{\underset{1}{\bullet}} \ar[r, "0"] & \cdots \ar[r, "0"] & \overset{\C}{\underset{i}{\bullet}} \ar[r, "\id"] & \cdots \ar[r, "\id"] & \overset{\C}{\underset{j}{\bullet}} \ar[r, "0"] & \cdots \ar[r, "0"] & \overset{0}{\underset{n+1}{\bullet}}
\end{tikzcd}.
\end{equation*}
\vspace{-10pt}

First, we observe that $M'$ can be written as the direct sum $M'=\bigoplus_{i=1}^{n+1}U_{i,n+1}$, and since $\dim_{\C}\Ext^1_{\Gamma'}(U_{k,l},U_{i,j})=1$ if and only if $k+1\leq i \leq l+1\leq j$ and zero otherwise (see for instance \cite{cerulli2017linear}[Section 3.1]), we conclude that $\Ext^1_{\Gamma'}(M',M')=0$.
\end{proof}

\begin{proposition}\label{prop:rigidM}
$M$ is a rigid representation of $(\Gamma, I)$.
\end{proposition}
\begin{proof}

\begin{comment}
\begin{equation*}
\begin{tikzcd}[]
\overset{\C}{\bullet} \ar[d, "\iota_{2,1}"] \\
\overset{\C^2}{\bullet}  \ar[d, "\iota_{3,2}"] \\
\overset{\C^3}{\bullet}  \ar[d, "\iota_{4,3}"] \\
... \ar[d, "\iota_{n+1,n}"] \\
\overset{\C^{n+1}}{\bullet} \\
\end{tikzcd}
\end{equation*}
\end{comment}

Consider $\Gamma'$ and $M'$ as in Lemma \ref{lemma:subquiver} and the functor
\begin{equation*}
    \Phi: \rep_{\C}(\Gamma')\to \rep_{\C}(\Gamma,I)
\end{equation*}
defined on $R\in \rep_{\C}(\Gamma')$ as follows. For all $i=1,\dots,n+1$ and $j=1,\dots,n$, we set $\Phi(R)_{i,j}=R_i$. For each arrow $i\to i+1$ in $\Gamma'$ and $j=1,\dots,n$, the map $\Phi(R)_{i,j}\to \Phi(R)_{i+1,j}$ is defined as the map $R_i\to R_{i+1}$. Finally, for each $i=1,\dots,n+1$ and $j=1,\dots,n-1$, the map $\Phi(R)_{i,j}\to \Phi(R)_{i,j+1}$ is $\id_{R_i}$. From the definition of $\Phi$, it follows that $\Phi(M')=M$.
As shown in \cite[Lemma 2.3, Lemma 2.5]{maksimau2019flag}, $\Phi$ is an exact, fully faithful functor that takes projective objects to projective objects. This implies (\cite[Corollary 2.6]{maksimau2019flag}) that $\Ext^i_{(\Gamma,I)}(\Phi(V),\Phi(W))\cong \Ext^i_{\Gamma'}(V,W) $ for every $V,W \in \rep_{\C}(\Gamma')$ and $i\geq 0$. In particular, we have $$\Ext^1_{(\Gamma,I)}(M,M)\cong \Ext^1_{\Gamma'}(M',M')=0.$$
\end{proof}

The rigidity of $M$, together with a few homological properties, can be exploited to deduce some geometric properties of the quiver Grassmannians $\Gr_{\bf{e}}(M)$ associated to $M$, for any dimension vector $\bf e$. For this purpose, we prove the following facts (see \cite[Section 5.4]{schiffler2014quiver} for the definitions of projective dimension and injective dimension of a module):

\begin{proposition}\label{prop:projinjdim}
Given $(\Gamma,I)$ and $M$ as above,
    \begin{enumerate}
        \item the projective dimension of $M$ is 0;
        \item the injective dimension of $M$ is 1.
    \end{enumerate}
\end{proposition}
\begin{proof}
 \begin{enumerate}
     \item The representation $M$ is a projective representation of $(\Gamma,I)$, and therefore has projective dimension equal to zero, because it is a direct sum of projective $(\Gamma,I)$-representations. Let us denote by $P(i,j)$ The indecomposable projective representation of $(\Gamma,I)$ at vertex $(i,j)$; then, $M= P(1,1)\oplus P(2,1) \oplus \dots \oplus P(n,1) \oplus P(n+1,1)$.
     
     \item  In order to show that the injective dimension of $M$ is one, we construct an injective resolution of $M$ . We define the injective representations $I_0$, $I_1$ of $(\Gamma,I)$ as the following sums of indecomposable injective representations:
    \begin{equation*}
        I_0=\bigoplus_{i=1}^{n+2} I(n+1,n),\quad I_1= \bigoplus_{i=1}^{n+1} I(i,n).
    \end{equation*}
    
    Then, the sequence 
    \begin{equation*}
        \begin{tikzcd}
        \varepsilon : 0 \arrow[r] & M \arrow[r, "f"] & I_0 \arrow[r,"g"] & I_1 \arrow[r] & 0
    \end{tikzcd}
    \end{equation*}
    where $f$ is the injective map embedding $M$ into $I_0$ and $g$ is the surjective map projecting $I_0$ onto $I_1$ (which implies $\im(f)=M=\ker(g)$) is a short exact sequence.
 \end{enumerate}   
\end{proof}

\begin{corollary}\label{cor:dimension}
Given $(\Gamma,I)$ and $M$ as above, the quiver Grassmannian $\Gr_{\bf{e}}(M)$, if not empty, is a smooth and irreducible projective variety for any dimension vector $\bf e$. Its dimension is $\langle \mathbf{e}, \boldsymbol{\dim} (M)-\mathbf{e} \rangle$.
\end{corollary}
\begin{proof}
As shown in Proposition \ref{prop:rigidM}, $M$ is a rigid representation of $(\Gamma,I)$, therefore the irreducibility of $\Gr_{\bf{e}}(M)$ follows directly from \cite[Proposition 3.8]{irelli2021cell}. In order to prove the remaining claims, we show that all hypotheses of \cite[Proposition 7.1]{irelli2013desingularization} hold. The representation $M$ has projective dimension zero, since it is a projective representation of $(\Gamma,I)$, and its injective dimension is one (see Proposition \ref{prop:projinjdim}).
It is straightforward to verify that the quotient algebra $\C\Gamma/I$ has global dimension two, since it can be realised as the tensor product of two well-known path algebras. Namely, we consider the path algebra of the cartesian product of an equioriented $\mathbb{A}_n$ quiver and an equioriented $\mathbb{A}_{n+1}$ quiver and take the quotient over the commutativity relations on all resulting squares. It is known that the global dimension of the path algebra of any type $\mathbb{A}_n$ quiver (for $n\geq 2$) is one (see, for instance, \cite[Section 2.2]{schiffler2014quiver}). Then, we apply \cite[Theorem 16]{auslander1955dimension} and obtain that the global dimension of $\C\Gamma/I$ is the sum of the global dimensions of the path algebras of the two quivers of type $\mathbb{A}_n$.
\end{proof}

\section{Recovering the Bott-Samelson resolution for Schubert varieties}\label{sec:desing}

We consider the quiver $(\Gamma,I)$ and its representation $M$ constructed in the previous section, and fix a permutation $w$ in $S_{n+1}$. %corresponding to a singular Schubert variety.
The conditions that define the elements $V\bm{.}$ in $X_w$ are of the form $\dim(F_p\cap V_q)\geq \#\set{k\leq q : w(k)\leq p}$, for $1\leq p,q \leq n+1$ (see Definition \ref{def:schubvar}). Notice that for $q=n+1$ and any $p$ these conditions are trivial, since $n+1$ is the dimension of the ambient space $\C^{n+1}$, and therefore it is enough to consider $q=1,\dots,n$.

Now we define the dimension vector ${\bf r}^w=(r^w_{i,j})$ for the quiver $(\Gamma,I)$ as
\begin{equation}\label{def:dimvecsing}
  r^w_{i,j}\coloneqq \#\set{k\leq j : w(k)\leq i}, \quad i=1,\dots,n+1,\quad j=1,\dots,n.
\end{equation}
Before introducing the Bott-Samelson resolution for Schubert varieties, let us make a few remarks about this definition for the dimension vector ${\bf r}^w$, in particular about how its entries change as we move from $w$ to permutations that are bigger than $w$ with respect to the Bruhat order in $S_{n+1}$.
The following lemma describes which (unique) row and which columns of the dimension vector are affected, and how they change, when we left-multiply by a simple transposition which increases by one the length of the permutation we are considering.

\begin{lemma}\label{lemma:changeinrow}
Consider $r^{\hat{w}}_{p,q}=\#\set{j\leq q : \hat{w}(j)\leq p}$ for $1 \leq p, q \leq n+1$ and a fixed $\hat{w}\in S_{n+1}$ (see Definition \ref{def:schubvar}). Then, for a simple transposition $s_i$ such that $\ell(s_i \hat{w})=\ell(\hat{w})+1$, the numbers $r^{s_i \hat{w}}_{p,q}=\#\set{j\leq q : s_i\hat{w}(j)\leq p}$ are given by
\begin{equation*}
    \begin{cases}
    r^{s_i \hat{w}}_{p,q}= r^{\hat{w}}_{p,q} - 1 & \text{ if } p=i \text{ and } q_i\leq q<q_{i+1}  \\
    r^{s_i \hat{w}}_{p,q}= r^{\hat{w}}_{p,q} & \text{ otherwise}
\end{cases}
\end{equation*}
where $q_i=\hat{w}^{-1}(i)$ and $q_{i+1}=\hat{w}^{-1}(i+1)$.
\end{lemma}

\begin{proof}
    It is straightforward to verify that, since $s_i$ only swaps $i$ and $i+1$, the count is not affected when $p\neq i$ or when $p=i$ and $q<q_i \; \vee \; q\geq q_{i+1}$.
    
    If $p=i$ and $q_i\leq q<q_{i+1}$, there is exactly one $j$ that satisfies $j\leq q \; \wedge \; \hat{w}(j)\leq p$ but not $j\leq q \; \wedge \; s_i\hat{w}(j)\leq p$, that is $j=q_i$, and so in this case the count decreases by one.
\end{proof}

\begin{example}\label{ex:changeinrow}
    We fix $\hat{w}=[34251]\in S_5$ and compute the corresponding dimension vector ${\bf r}^{\hat{w}}$ according to the definition given in \eqref{def:dimvecsing}:
    \begin{equation*}
{\bf r}^{\hat{w}}=\begin{pmatrix}
    0& 0& 0& 0\\
    0& 0& 1& 1\\
    1& 1& 2& 2\\
    1& 2& 3& 3\\
    1& 2& 3& 4
\end{pmatrix}.
\end{equation*}
Then we apply $s_3$, obtaining $w\coloneqq s_3 \hat{w}=[43251]$, and we know from Lemma \ref{lemma:changeinrow} that ${\bf r}^w$ differs from ${\bf r}^{\hat{w}}$ only at entry $r^w_{3,1}$:
 \begin{equation*}
{\bf r}^w=\begin{pmatrix}
    0& 0& 0& 0\\
    0& 0& 1& 1\\
    \textcolor{red}{0}& 1& 2& 2\\
    1& 2& 3& 3\\
    1& 2& 3& 4
\end{pmatrix}.
\end{equation*}
\end{example}

\begin{remark}
    An important consequence of Lemma \ref{lemma:changeinrow} is that some information about the reduced decompositions of $w$ can be read directly off the corresponding dimension vector. In Example \ref{ex:changeinrow}, for instance, we can compare ${\bf r}^w$ to the dimension vector corresponding to the identity in $S_5$:
    \begin{equation*}
{\bf r}^{\id}=\begin{pmatrix}
    1& 1& 1& 1\\
    1& 2& 2& 2\\
    1& 2& 3& 3\\
    1& 2& 3& 4\\
    1& 2& 3& 4
\end{pmatrix}
\end{equation*}
and observe that $r^w_{2,2}=r^{\id}_{2,2}-2$. By Lemma \ref{lemma:changeinrow}, this can only happen if the simple transposition $s_2$ appears at least two times in any reduced decomposition of $w$. Similarly, we deduce that $s_1$, $s_3$ and $s_4$ appear at least one time in any reduced decomposition of $w$.
The converse is also true: if a simple transposition $s_i$ appears $k$ times in all reduced decompositions of a given permutation $w$ (that is, there are $k$ instances of $s_i$ that are not part of any braid $s_i s_{i+1} s_i$ or $s_{i+1} s_i s_{i+1}$), then there exists an entry in the $i$-th row of the dimension vector ${\bf r}^w$ that has decreased by $k$ from its value in ${\bf r}^{\id}$. We do not include a proof of this statement as it is not relevant to the purpose of this section, but an idea of the strategy can be found in the proof of Theorem \ref{them:geomdecomp}, since knowing that these $k$ instances of $s_i$ are not part of any braid allows us to describe which simple transpositions can appear between them.
\end{remark}

In order to show that the quiver Grassmannian $\Gr_{{\bf r}^w}(M)$ is isomorphic to certain Bott-Samelson resolutions of $X_w$, we recall the following definition of Bott-Samelson varieties:

\begin{definition}(\cite[Definition 3.1]{hudson2020stability})\label{def:BS}
Given a permutation $w\in S_{n+1}$ of length $N$ and a reduced decomposition $w=s_{i_N}\cdots s_{i_1}$, the \Def{Bott-Samelson variety} $\bs(s_{i_N}\cdots s_{i_1})$ is a subvariety of $(\Flag_{n+1})^{N}$ defined as follows:

\begin{equation*}
    \begin{aligned}
        \bs(s_{i_N}\cdots s_{i_1})= \{& (V^0\bm{.},V^1\bm{.},\dots,V^{N}\bm{.})\in (\Flag_{n+1})^{N} : V^{k-1}_i=V^k_i, \forall k=1,\dots,N,\\& \forall i =1,\dots, n, i\neq i_k \}
    \end{aligned}
\end{equation*}

where $V^0\bm{.}=F\bm{.}=\langle b_1\rangle \subseteq  \langle b_1, b_2\rangle \subseteq \dots \subseteq \langle b_1, b_2,\dots, b_{n+1} \rangle$, the standard flag in $\Flag_{n+1}$.
\end{definition}

\begin{example}\label{ex:BS}
    We fix again the permutation $w=[43251]\in S_5$  from Example \ref{ex:changeinrow} and its reduced decomposition $w=s_1s_2s_3s_1s_2s_1s_4$.
    The elements $V^{\bullet}\bm{.}$ of $\bs(s_1s_2s_3s_1s_2s_1s_4)$ are given by tuples of seven complete flags, each living in $\Flag_5$, of the following form:
    \begin{equation*}
    \begin{aligned}
        & \langle b_1\rangle \subseteq  \langle b_1, b_2\rangle \subseteq \langle b_1, b_2, b_3 \rangle \subseteq V^1_4, \quad && V^4_1\subseteq V^3_2 \subseteq V^5_3 \subseteq V^1_4,\\
        &V^2_1\subseteq \langle b_1, b_2\rangle \subseteq \langle b_1, b_2, b_3 \rangle \subseteq V^1_4, && V^4_1\subseteq V^6_2 \subseteq V^5_3 \subseteq V^1_4,\\
        &V^2_1\subseteq V^3_2 \subseteq \langle b_1, b_2, b_3 \rangle \subseteq V^1_4, && V^7_1\subseteq V^3_2 \subseteq V^5_3 \subseteq V^1_4,\\
        &V^4_1\subseteq V^3_2 \subseteq \langle b_1, b_2, b_3 \rangle \subseteq V^1_4.
    \end{aligned}   
    \end{equation*}
\end{example}

\begin{remark}
    The Bott-Samelson varieties corresponding to different reduced decompositions of the same permutation $w$ are birational (see \cite{anderson2023equivariant}[Chapter 18, Lemma 2.1]). 
    Therefore, to show that the quiver Grassmannian $\Gr_{{\bf r}^w}(M)$ is birational to any Bott-Samelson resolution of $X_w$, it is enough to pick an opportune reduced decomposition of $w$.
\end{remark}

\begin{definition}
    Let $w\in S_{n+1}$ and denote by $R=(R_{p,q})$, for $p=1,\dots,n+1$ and $q=1,\dots,n$, an element of $\Gr_{{\bf r}^w}(M)$. We call a reduced decomposition $w=s_{i_N}\cdots s_{i_2}s_{i_1}$ \Def{ geometrically compatible } if $[i_{N},\dots,i_2,i_1]=[r^w_{p,q}]$ for all the $p,q$ such that
    \begin{equation}\label{eq:freesubspaces}
        \begin{cases}
            r^w_{p,q}<p \\
            r^w_{p,q}>r^w_{p-1,q} \\
            r^w_{p,q}>r^w_{p,q-1}
        \end{cases},
        \end{equation}
    where the notations $[i_{N},\dots,i_2,i_1]$ and $[r^w_{p,q}]$ indicate nonordered multisets. 
\end{definition}

\begin{remark}
The $r^w_{p,q}$ that satisfy the conditions listed in \eqref{eq:freesubspaces} are the dimensions of exactly those subspaces $R_{p,q}$ of $\C^p$ that are not trivial and do not coincide with a subspace to their left or above them.
Fixing an element $R$ in the quiver Grassmannian $\Gr_{{\bf r}^w}(M)$ means precisely to make a choice for all such subspaces $R_{p,q}$. Hence the name "geometrically compatible" decomposition: it is a reduced decomposition of $w$ from which we can read the dimensions of all the subspaces that are relevant to determine $R$.
\end{remark}

Because of the geometrical significance of the conditions given in \eqref{eq:freesubspaces}, we will interchangeably refer to the $r^w_{p,q}$ and to the $R_{p,q}$ that satisfy these conditions.

\begin{example}\label{ex:twodecomps}
 We consider $w=[43251]$ and the corresponding dimension vector ${\bf r}^w$ as in \ref{ex:changeinrow}. Given any subrepresentation $R$ in $\Gr_{{\bf r}^w}(M)$, the subspaces whose dimensions satisfy all conditions in \eqref{eq:freesubspaces} are $R_{2,3},R_{3,2},R_{4,1},R_{3,3},R_{4,2},R_{4,3}$ and $R_{5,4}$. Their dimensions are, in order, 1,1,1,2,2,3, and 4, so a geometrically compatible decomposition of $w$ contains three $s_1$, two $s_2$, one $s_3$ and one $s_4$. The decomposition $w=s_1s_2s_3s_1s_2s_1s_4$ considered in \ref{ex:BS} is geometrically compatible, while, for instance, $w=s_3s_1s_2s_1s_3s_2s_4$ is not.
\end{example}

To show that all permutations admit a geometrically compatible decomposition, we first characterise the subspaces appearing in \eqref{eq:freesubspaces}, that is, what follows from the fact that a certain $R_{p,q}$ is not a trivial subspace of $\C^p$ in terms of the reduced decompositions of $w$.
Recall that the length of a permutation $w$ can be equivalently  defined as the number of inversions appearing in $w$ or as the number of simple transpositions that form any reduced decomposition of $w$.

\begin{lemma}\label{lemma:correctlength}
Given $w\in S_{n+1}$ and a subrepresentation $R=(R_{p,q})$ in $\Gr_{{\bf r}^w}(M)$, with $p=1,\dots,n+1$ and $q=1,\dots,n$,  the number of subspaces that satisfy all conditions in \eqref{eq:freesubspaces} is exactly the length of $w$.
\end{lemma}

\begin{proof}
We denote by $N$ the length of $w$ and write $w=s_{i_N}\cdots s_2s_1$. By Lemma \ref{lemma:changeinrow}, the left-multiplication of each $s_{i_k}$ results in a new subspace (namely $R_{i_k+1,q_{i_k}}$, for $q_{i_k}$ as in the notation of Lemma \ref{lemma:changeinrow}) satisfying the conditions in \eqref{eq:freesubspaces}.
On the other hand, applying one simple transposition cannot cause two new subspaces to satisfy the conditions in \eqref{eq:freesubspaces}, because all the affected entries of the dimension vector decrease by the same amount. 
\end{proof}

\begin{lemma}\label{lemma:alldistinct}
Let $w\in S_{n+1}$ and $R$ any subrepresentation in $\Gr_{{\bf r}^w}(M)$. If the $r^w_{p,q}$ that satisfy the conditions in \eqref{eq:freesubspaces} are all distinct, then $w$ admits a geometrically compatible decomposition.
\end{lemma}

\begin{proof}
In order for a subspace of dimension $d$ to satisfy the conditions in \eqref{eq:freesubspaces}, the simple transposition $s_d$ must appear at least once in any reduced decomposition of $w$. If this weren't the case, by Lemma \ref{lemma:changeinrow} the $d$-th row of ${\bf r}^w$ would be equal to the $d$-th row of ${\bf r}^{\id}$, which would imply that all subspaces appearing in $R$ of dimension $d$ have to coincide with $\C^d$ - and therefore do not satisfy the conditions in \eqref{eq:freesubspaces}.
The result then follows immediately from Lemma \ref{lemma:correctlength}.
\end{proof}

\begin{remark}\label{remark:alldistinct}
A straightforward consequence of Lemma \ref{lemma:changeinrow} is that if the reduced decompositions of $w\in S_{n+1}$ consist of all distinct simple transpositions, then they are geometrically compatible. As shown in the lemma, each of these transpositions $s_{i_k}$ affects the corresponding row of the dimension vector, resulting in the subspace $R_{i_k+1,q_i}$ (which has dimension $i_k$) satisfying the conditions in \eqref{eq:freesubspaces}.
\end{remark}

We recall the following notation from Lemma \ref{lemma:changeinrow}: if we left-multiply a permutation $w$ by a simple transposition $s_j$, we denote by $q_j$ the pre-image $w^{-1}(j)$ of $j$ via $w$.

\begin{lemma}\label{lemma:dimofspaces}
Let $w\in S_{n+1}$ with reduced decomposition $w=s_{i_N} \dots s_{i_1}$, $s_j$ a simple transposition such that $\ell(s_j w)=\ell(w)+1$, and $R_{j+1,q_j}$ the subspace that satisfies the conditions in \eqref{eq:freesubspaces} if $R$ is any subrepresentation in $\Gr_{{\bf r}^{s_j w}}(M)$ (but does not satisfy them if $R$ is in $\Gr_{{\bf r}^w}(M)$). Then, the dimension of $R_{j+1,q_j}$ is $\hat{j}$ for some $\hat{j}\leq j$. In particular, $\hat{j}<j$ can only happen if all reduced decompositions of $s_j w$ are of the form $s_j w=s_j s_{i_N} \dots s_{i_k} \dots s_{i_1}$, where $i_k=j$ for some $k$ such that $i_t\neq j+1$ for all $k<t<N$.
\end{lemma}

\begin{proof}
The first statement is almost straightforward. A subspace $R_{p,q}$ can satisfy the conditions in \eqref{eq:freesubspaces} only if $p\geq \dim(R_{p,q})+1$, and we know from Lemma \ref{lemma:changeinrow} that the only effect of $s_j$ on the corresponding dimension vector is to decrease certain entries in row $j$ by one. By Definition \ref{def:dimvecsing} of the dimension vector, all entries are bounded by their corresponding numbers of row and column (which implies that a dimension $j'$ can only appear from row $j'$ downwards), and so the dimension of $R_{j+1,q_j}$ cannot be greater than $j$.

For the second statement, we know that an index $k$ such that $i_k=j$ exists: as stated in Remark \ref{remark:alldistinct}, if the simple transpositions appearing in the reduced decomposition of $s_j w$ were all distinct, then the dimension of $R_{j+1,q_j}$ would be $j$. Then, we suppose that $s_{j+1}$ occurs between these two instances of $s_j$ and look at which entries of the dimension vector decrease when $w$ is left-multiplied by $s_j$. According to Lemma \ref{lemma:changeinrow}, the entries in columns $q_j$ and $q_j + 1$ (and possibly more) would decrease by one, meaning that the subspace $R_{j+1,q_j}$, for $R\in \Gr_{{\bf r}^{s_j w}}(M)$, cannot satisfy the conditions in \eqref{eq:freesubspaces}, which contradicts the assumption.
\end{proof}

\begin{remark}\label{remark:braidmoves}
Lemma \ref{lemma:dimofspaces} provides a characterisation of when braid moves are possible in a decomposition of $w\in S_{n+1}$ in terms of the dimensions of the subspaces $R_{p,q}$ that satisfy the conditions in \eqref{eq:freesubspaces}, for $R\in\Gr_{{\bf r}^w}(M)$. The second statement in Lemma \ref{lemma:dimofspaces} implies that if a transposition $s_i$ appears $k$ times in all reduced decompositions of $w$ (i.e. these $k$ instances of $s_i$ are not part of any braid move) then there are (at least) $k$ subspaces $R_{p,q}$ of dimension $i$ that satisfy the conditions in \eqref{eq:freesubspaces}. On the other hand, if we apply $s_j$ after $w$ and obtain a reduced decomposition of $s_j w$ that is not geometrically compatible, we know that it is possible to perform a braid move on $s_j s_{j-1} s_j$. This follows from the fact that we can move $s_j$ to the right via commutation until we find an instance of $s_{j-1}$, and similarly move the second instance of $s_j$ to the left until $s_{j-1}$ ($s_{j+1}$ cannot occur in between by Lemma \ref{lemma:dimofspaces}).

For instance, we saw in Example \ref{ex:twodecomps} a reduced decomposition for $w=[43251]$ that is not geometrically compatible: $w=s_3s_1s_2s_1s_3s_2s_4$. We obtain $w=s_1s_3s_2s_3s_1s_2s_4$ by commutation on the two occurrences of $s_3$, then perform a braid move as described above and get $w=s_1s_2s_3s_2s_1s_2s_4$. Finally, we perform a braid move on $s_2 s_1 s_2$ and obtain the geometrically compatible decomposition of $w$ shown in Example \ref{ex:twodecomps}: $w=s_1s_2s_3s_1s_2s_1s_4$.
\end{remark}

\begin{theorem}\label{them:geomdecomp}
    All permutations admit a geometrically compatible decomposition.
\end{theorem}
\begin{proof}
Let us denote by $t$ the total number of repetitions in a given reduced decomposition of a permutation (i.e. how many times any simple transposition is repeated) and by $m$ the difference between the length of $w$ and $t$.
We prove the statement by double induction on $m$ and $t$.

The base case of the induction ($m=1$ and $t=0$) and the induction step on $m$ ($m\geq 1$ and $t=0$) follow directly from Remark \ref{remark:alldistinct}: a reduced decomposition of $w$ without repetitions consists of distinct simple transpositions, and is therefore geometrically compatible. For the induction step on $t$ we show that, if a permutation with $t\geq 0$ repetitions admits a geometrically compatible decomposition, then a permutation with $t+1$ repetitions admits a geometrically compatible decomposition (for any $m\geq 1$).
Let $w=s_{i_N}\dots s_{i_1}$ with $t$ repetitions be a geometrically compatible decomposition of $w$. We denote by $d_i$ the number of subspaces $R_{p,q}$ of dimension $i$, for $R\in \Gr_{{\bf r}^w}(M)$, that satisfy the conditions in \eqref{eq:freesubspaces}. Since the fixed decomposition of $w$ is geometrically compatible, we have $\# s_i=d_i$ for all $i$. Let then $w'\coloneqq s_jw=s_j s_{i_N}\dots s_{i_1}$ such that $w'$ has $t+1$ repetitions, which means $j=i_k$ for some $k$, and such that $\ell(w')=\ell(w)+1$.
If the corresponding new free subspace appearing in $R$ has dimension $j$, then this reduced decomposition of $w'$ is already geometrically compatible. If not, then by Lemma \ref{lemma:dimofspaces} the dimension of the new free subspace must be $\hat{j}<j$, and therefore $d_{\hat{j}}$ has increased by one. As described in Remark\ref{remark:braidmoves}, we move $s_j$ via commutation and perform a braid move: $w' = s_j s_{i_N}\dots s_{i_1}=s_{i_N}\dots s_j s_{j-1} s_j \dots s_{i_1}=s_{i_N}\dots s_{j-1} s_j s_{j-1} \dots s_{i_1}$, which decreases by one $\# s_j$ (the number of occurrences of $s_j$) and increases by one $\# s_{j-1}$.
Now, if $\hat{j}=j-1$ we have again $\# s_i=d_i$ for each $i$, meaning that this reduced decomposition of $w'$ is geometrically compatible. Otherwise, we denote by $\hat{w}$ the subword of $w$ starting from the second instance of $s_{j-1}$: $\hat{w}=s_{j-1} \dots s_{i_1}$ and observe that $\hat{w}$ has $t$ repetitions. Therefore, by the induction hypothesis, $\hat{w}$ admits a geometrically compatible decomposition. We know that the current decomposition of $\hat{w}$ is not geometrically compatible, because the number of $s_{\hat{j}}$ appearing in $\hat{w}$ is $d_{\hat{j}}-1$.
The geometrically compatible decomposition of $\hat{w}$ must then be obtained by performing a sequence of braid moves until the braid $s_{\hat{j}+1} s_{\hat{j}} s_{\hat{j}+1}=s_{\hat{j}}s_{\hat{j}+1}s_{\hat{j}}$. Each braid move decreases by one the number of $s_{l+1}$ and increases by one the number of $s_l$, for $j-2\leq l\leq \hat{j}$. Since the number of all other transpositions appearing in $\hat{w}$ (and in $w'$) is not changed during this process, in the end we get $\# s_i=d_i$ for all $i$, which means that we obtained a geometrically compatible decomposition of $w'$.
\end{proof}

\begin{theorem}\label{thm:BSisom}
    Given a permutation $w\in S_{n+1}$ and a geometrically compatible decomposition $w=s_{i_N}\cdots s_1$, the Bott-Samelson resolution $\bs(s_{i_N}\cdots s_{i_1}) $ of the Schubert variety $X_w$ is isomorphic to the quiver Grassmannian $\Gr_{{\bf r}^w}(M)$.
\end{theorem}
\begin{proof}
    Given a geometrically compatible decomposition $w=s_{i_N}\cdots s_1$, we define a map $\varphi_w$ according to the correspondence between the ordered set of indices of the transpositions appearing in $w$ and the vector space $R_{p,q}$ for any \mbox{$R\in \Gr_{{\bf r}^w}(M)$:}
    \begin{equation}\label{eq:mapphi}
    \begin{aligned}
        \varphi_w: \{ i_N,\dots,i_1 \} & \to \{ n+1 \}\times \{ n \}\\
        i_k & \mapsto (p(k),q(k))\coloneqq (i_k+1+n_k,i_k+m_k)
    \end{aligned}
    \end{equation}
    with $n_k\coloneqq \#\{j: j<k, i_j=i_k \} $  and $m_k\coloneqq \# \{j: j>k, q_{i_j}\leq q(k) < q_{i_j+1} \}$, where $q_{i_j}$ and $q_{i_j+1}$ are defined as in Lemma \ref{lemma:changeinrow}.

    We prove the statement by induction on the length of $w\in S_{n+1}$. For $w=\id$, the corresponding Bott-Samelson resolution and quiver Grassmannian coincide since they consist of a single point. We then denote $w'=s_{i_{N-1}}\cdots s_{i_1}$ and assume $\bs(s_{i_{N-1}}\cdots s_{i_1}) \cong \Gr_{{\bf r}^{w'}}(M)$, where the isomorphism is given by $\varphi_{w'}$. This means that the explicit correspondence between an element $V^{\bullet}\bm{.}\in \bs(s_{i_{N-1}}\cdots s_{i_1})$ and a subrepresentation $R'\in \Gr_{{\bf r}^{w'}}(M)$ is $V^{b}_a=R'_{p(b),q(b)}$, therefore they are defined by the same inclusion conditions. We now consider $w=s_{i_N} w'$ such that $\ell(w)=\ell(w')+1$. The image of $i_N$ via $\varphi_w$ is $(p(N),q(N))=(i_N+1+n_N,i_N+m_N)$: we need to show that the subspace $R_{p(N),q(N)}$ is isomorphic to the subspace $V^N_{i_N}$, whose defining conditions are $V^{b}_a\subseteq V^N_{i_N}\subseteq V^{c}_d$ for $b,c<N$ and $a<i_N<d$. By the induction hypothesis, the subspaces $R_{p(b),q(b)}$ and $R_{p(c),q(c)}$ realise respectively $V^{b}_a$ and $V^{c}_d$ for all such $a,b,c,d$. We observe that the dimension of $R_{p(N),q(N)}$ is $i_N$ due to the choice of a geometrically compatible decomposition of $w$. The statement then follows from the fact that $R$ is a subrepresentation of $M$, which implies $R_{p,q}\subseteq R_{p(N),q(N)}\subseteq R_{p',q'}$ for all $p\leq p(N), q\leq q(N)$, $p'\geq p(N), q'\geq q(N)$ and so in particular for $p=p(b), q=q(b)$, $p'=p(c), q'=q(c)$.
\end{proof}

\begin{corollary}
    Since the Bott-Samelson resolutions corresponding to different reduced decompositions of the same permutation are birational, they are all birational to $\Gr_{{\bf r}^w}(M)$.
\end{corollary}

\begin{example}
 Given a permutation $w'$, the map $\varphi_{w}$ defined in Equation \eqref{eq:mapphi} describes explicitly which subspace $R_{p,q}$, for $R\in \Gr_{{\bf r}^{w}}(M)$, becomes a nontrivial subspace of $\C^{i}$ when $s_i$ is applied to $w'$, with $\ell(w)=\ell(w')+1$. Consider, for instance, the geometrically compatible decomposition $w=s_1s_2s_3s_1s_2s_1s_4$ of Example \ref{ex:twodecomps}, where $s_1$ appears three times, as $s_{i_2}, s_{i_4}$ and $s_{i_7}$. The images of $i_2, i_4$ and $i_7$ via the map $\varphi_w$ defined in \eqref{eq:mapphi} are $(i_2+1+n_2,i_2+m_2)=(2,3)$, $(i_4+1+n_4,i_4+m_4)=(3,2)$ and $(i_7+1+n_7,i_7+m_7)=(4,1)$. The one-dimensional subspaces $R_{2,3}$, $R_{3,2}$ and $R_{4,1}$, which correspond respectively to the subspaces $V^2_1$, $V^4_1$ and $V^7_1$ considered in \ref{ex:BS}, can be visualised at the red vertices of $(\Gamma, I)$:

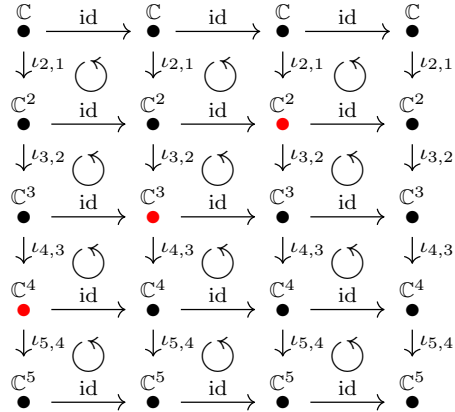
\begin{equation*}
\begin{tikzcd}[row sep=small]
\overset{\C}{\bullet} \ar[r, "\id"] \ar[d, "\iota_{2,1}"]  & \overset{\C}{\bullet} \ar[r, "\id"] \ar[d, "\iota_{2,1}"] & \overset{\C}{\bullet} \ar[r, "\id"] \ar[d, "\iota_{2,1}"] & \overset{\C}{\bullet} \ar[d, "\iota_{2,1}"] \\
\overset{\C^2}{\bullet} \ar[ur, phantom, "\scalebox{1.5}{$\circlearrowleft$}"] \ar[r, "\id"] \ar[d, "\iota_{3,2}"] & \overset{\C^2}{\bullet} \ar[ur, phantom, "\scalebox{1.5}{$\circlearrowleft$}"] \ar[r, "\id"] \ar[d, "\iota_{3,2}"]  & \overset{\C^2}{\textcolor{red}{\bullet}} \ar[ur, phantom, "\scalebox{1.5}{$\circlearrowleft$}"] \ar[r, "\id"] \ar[d, "\iota_{3,2}"] & \overset{\C^2}{\bullet} \ar[d, "\iota_{3,2}"] \\
\overset{\C^{3}}{\bullet} \ar[ur, phantom, "\scalebox{1.5}{$\circlearrowleft$}"] \ar[r, "\id"] \ar[d, "\iota_{4,3}"] & \overset{\C^{3}}{\textcolor{red}{\bullet}} \ar[ur, phantom, "\scalebox{1.5}{$\circlearrowleft$}"] \ar[r, "\id"] \ar[d, "\iota_{4,3}"] & \overset{\C^{3}}{\bullet} \ar[ur, phantom, "\scalebox{1.5}{$\circlearrowleft$}"] \ar[r, "\id"] \ar[d, "\iota_{4,3}"] & \overset{\C^{3}}{\bullet} \ar[d, "\iota_{4,3}"]\\
\overset{\C^{4}}{\textcolor{red}{\bullet}} \ar[ur, phantom, "\scalebox{1.5}{$\circlearrowleft$}"] \ar[r, "\id"] \ar[d, "\iota_{5,4}"] & \overset{\C^{4}}{\bullet} \ar[ur, phantom, "\scalebox{1.5}{$\circlearrowleft$}"] \ar[r, "\id"] \ar[d, "\iota_{5,4}"] & \overset{\C^{4}}{\bullet} \ar[ur, phantom, "\scalebox{1.5}{$\circlearrowleft$}"] \ar[r, "\id"] \ar[d, "\iota_{5,4}"] & \overset{\C^{4}}{\bullet} \ar[d, "\iota_{5,4}"]\\
\overset{\C^{5}}{\bullet} \ar[ur, phantom, "\scalebox{1.5}{$\circlearrowleft$}"] \ar[r, "\id"] & \overset{\C^{5}}{\bullet} \ar[ur, phantom, "\scalebox{1.5}{$\circlearrowleft$}"] \ar[r, "\id"] & \overset{\C^{5}}{\bullet} \ar[ur, phantom, "\scalebox{1.5}{$\circlearrowleft$}"] \ar[r, "\id"] & \overset{\C^{5}}{\bullet}\\
\end{tikzcd}
\end{equation*}
\end{example}

\section{Realisation of smooth Schubert varieties}\label{sec:smooth}

In Section \ref{sec:desing}, we recovered the Bott-Samelson resolution for Schubert varieties by defining the dimension vector ${\bf r}^w$ for the quiver $(\Gamma,I)$ as
\begin{equation*}
    r^w_{i,j} = \#\set{k\leq j : w(k)\leq i}
\end{equation*}
 for all $i,j$. In this section we give a construction for a different dimension vector for the quiver $(\Gamma,I)$, denoted by ${\bf e}^w$, and show how the corresponding quiver Grassmannian realises the Schubert variety $X_w$ if it is smooth, i.e. if $w$ is pattern-avoiding.
We recall from Section \ref{sec:schubvar} that a permutation $w\in S_{n+1}$ corresponds to a smooth Schubert variety if and only if it avoids the patterns $[4231]$ and $[3412]$, and that this is equivalent to $X_w$ being defined by non-crossing inclusions (see Definition \ref{def:inclusions}).

Consider again the quiver $(\Gamma,I)$ and its representation $M$ constructed in Section \ref{sec:quiver}, and fix a permutation $w$ in $S_{n+1}$ that avoids the patterns $[4231]$ and $[3412]$. %The conditions that define the elements $V\bm{.}$ in $X_w$ are of the form $\dim(F_p\cap V_q)\geq \#\set{k\leq q : w(k)\leq p}$, for $1\leq p,q \leq n+1$ (see Definition \ref{def:schubvar}).
For $i=1,\dots,n+1$ and $j=1,\dots,n$, we now define the dimension vector ${\bf e}^w=(e^w_{i,j})$ for the quiver $(\Gamma,I)$ as:
\begin{equation}\label{def:dimvec}
\begin{cases}
    e^w_{i,j}\coloneqq r^w_{i,j} & \text{ if } r^w_{i,j} =\min \{ i,j \} \\
    & \text{ or } r^w_{i,j} =0\\
    e^w_{i,j}\coloneqq \max \{ e^w_{i-1,j}, e^w_{i,j-1}  \} & \text{ if } 0 < r^w_{i,j} <\min \{ i,j \}
\end{cases}.
\end{equation}
Notice that the value of $r^w_{1,1}$ is either 0 or 1 (according to $w$) and falls therefore under the first case of Definition \eqref{def:dimvec}, meaning that $e^w_{1,1}$ is well-defined.

\begin{comment}
    \begin{equation}\label{def:dimvec}
\begin{cases}
    e^w_{i,j}\coloneqq \#\set{k\leq j : w(k)\leq i} & \text{ if } \#\set{k\leq j : w(k)\leq i}=\min \{ i,j \} \\
    & \text{ or } \#\set{k\leq j : w(k)\leq i}=0\\
    e^w_{i,j}\coloneqq \max \{ e^w_{i-1,j}, e^w_{i,j-1}  \} & \text{ if } \#\set{k\leq j : w(k)\leq i} <\min \{ i,j \}
\end{cases}.
\end{equation}
\end{comment}

\begin{example}\label{ex:smoothvar}
We compute the conditions defining the flags $V\bm{.}$ in $X_w$ for $w=[65124837] \in S_8$ according to Definition \ref{def:schubvar}, denoting $\dim(F_p\cap V_q)$ by $d_{p,q}$:
\begin{center}
\begin{tabular}{c c c c c c c c}
     $d_{1,1}\geq 0$  & $d_{1,2}\geq 0$ & $d_{1,3}\geq 1$ & $d_{1,4}\geq 1$ & $d_{1,5}\geq 1$ & $d_{1,6}\geq 1$ & $d_{1,7}\geq 1$ & $d_{1,8}\geq 1$ \\
     
     $d_{2,1}\geq 0$  & $d_{2,2}\geq 0$ & $d_{2,3}\geq 1$ & $d_{2,4}\geq 2$ & $d_{2,5}\geq 2$ & $d_{2,6}\geq 2$ & $d_{2,7}\geq 2$ & $d_{2,8}\geq 2$ \\

     $d_{3,1}\geq 0$  & $d_{3,2}\geq 0$ & $d_{3,3}\geq 1$ & $d_{3,4}\geq 2$ & $d_{3,5}\geq 2$ & $d_{3,6}\geq 2$ & $d_{3,7}\geq 3$ & $d_{3,8}\geq 3$ \\

     $d_{4,1}\geq 0$  & $d_{4,2}\geq 0$ & $d_{4,3}\geq 1$ & $d_{4,4}\geq 2$ & $d_{4,5}\geq 3$ & $d_{4,6}\geq 3$ & $d_{4,7}\geq 4$ & $d_{4,8}\geq 4$ \\

     $d_{5,1}\geq 0$  & $d_{5,2}\geq 1$ & $d_{5,3}\geq 2$ & $d_{5,4}\geq 3$ & $d_{5,5}\geq 4$ & $d_{5,6}\geq 4$ & $d_{5,7}\geq 5$ & $d_{5,8}\geq 5$ \\

     $d_{6,1}\geq 1$  & $d_{6,2}\geq 2$ & $d_{6,3}\geq 3$ & $d_{6,4}\geq 4$ & $d_{6,5}\geq 5$ & $d_{6,6}\geq 5$ & $d_{6,7}\geq 6$ & $d_{6,8}\geq 6$ \\

     $d_{7,1}\geq 1$  & $d_{7,2}\geq 2$ & $d_{7,3}\geq 3$ & $d_{7,4}\geq 4$ & $d_{7,5}\geq 5$ & $d_{7,6}\geq 5$ & $d_{7,7}\geq 6$ & $d_{7,8}\geq 7$ \\

     $d_{8,1}\geq 1$  & $d_{8,2}\geq 2$ & $d_{8,3}\geq 3$ & $d_{8,4}\geq 4$ & $d_{8,5}\geq 5$ & $d_{8,6}\geq 6$ & $d_{8,7}\geq 7$ & $d_{8,8}\geq 8$ \\
\end{tabular}
\end{center}
Since $w$ avoids the patterns $[4231]$ and $[3412]$, $X_w$ is smooth and defined by non-crossing inclusions. These inclusions, which follow from the inequalities above, are:
\begin{equation}\label{eq:conditions}
\begin{aligned}
& V_1 \subset F_6, V_2 \subset F_6, F_1\subset V_3 \subset F_6, F_2\subset V_4\subset F_6, \\
& F_2\subset V_5\subset F_6, F_2\subset V_6, F_6\subset V_7.
\end{aligned}
\end{equation}
The corresponding dimension vector ${\bf e}^w$ obtained from \eqref{def:dimvec} is
\begin{equation*}
{\bf e}^w=\begin{pmatrix}
    0& 0& 1& 1& 1& 1& 1\\
    0& 0& 1& 2& 2& 2& 2\\
    0& 0& 1& 2& 2& 2& 3\\
    0& 0& 1& 2& 2& 2& 4\\
    0& 0& 1& 2& 2& 2& 5\\
    1& 2& 3& 4& 5& 5& 6\\
    1& 2& 3& 4& 5& 5& 6\\
    1& 2& 3& 4& 5& 6& 7\\
\end{pmatrix}.
\end{equation*}
By reading each entry $e^w_{i,j}$ as the dimension of the intersection 
$F_p\cap V_q$ and comparing ${\bf e}^w$ with the defining conditions in \eqref{eq:conditions}, we see how ${\bf e}^w$ encodes the same information on $V\bm{.}$.
\end{example}

\begin{theorem}\label{thm:isom}
If $w\in S_{n+1}$ avoids the patterns $[4231]$ and $[3412]$, the quiver Grassmannian $\Gr_{{\bf e}^w}(M)$ is isomorphic to the Schubert variety $X_w$. The isomorphism is given by
\begin{gather}\label{eq:mapsmooth}
\begin{aligned}
\psi: \Gr_{{\bf e}^w}(M) &\to X_w \\
N &\mapsto N\bm{.}
\end{aligned}
\end{gather}
where $N\bm{.}=N_{n+1,1}\subseteq N_{n+1,2}\subseteq\dots\subseteq N_{n+1,n}$.
\end{theorem}
\begin{proof}
By the definition of $M$ and ${\bf e}^w$, we have
\begin{equation*}
N_{n+1,j}\subseteq N_{n+1,j+1},\; \dim(N_{n+1,j})=j  
\end{equation*}
for all $j$, implying $N\bm{.}\in \Flag_{n+1}$.
Since $w$ avoids the patterns $[4231]$ and $[3412]$, all flags $V\bm{.}=V_1\subseteq V_2\subseteq \dots \subseteq V_n$ in $X_w$ are defined by conditions of the following form: for each $q\in \{1,\dots,n\}$, $V_q$ is defined by $F_{p'_q}\subseteq V_q\subseteq F_{p_q}$ for some $p_q, p'_q$.
These conditions are equivalent, respectively, to $\dim(F_{p'_q} \cap V_q)= \min(p'_q,q)=p'_q $ and $\dim(F_{p_q} \cap V_q)= \min(p_q,q)=q$.
The definition of the dimension vector ${\bf e}^w$ (in the first line of \eqref{def:dimvec}) imposes on $N\bm{.}$ exactly these conditions, meaning that $F_{p'_q} \subseteq N_{n+1,q} \subseteq F_{p_q}$ for all $q$ and the corresponding $p'_q, p_q$.
The statement follows from the fact that, whenever the condition $\dim(F_i\cap V_j)\geq \#\set{k\leq j : w(k)\leq i}$ is not defining for $V\bm{.}$ (i.e. it is redundant), the corresponding subspace $N_{i,j}$ in $N\bm{.}$ is set to either $N_{i-1,j}$ or $N_{i,j-1}$ (second line of \eqref{def:dimvec}).
\end{proof}

\begin{remark}
 The flag variety $\Flag_{n+1}$ can be defined equivalently as the quotient $G/B$, where $G=\GL_{n+1}$ and $B\subset G$ is the Borel subgroup of upper-triangular matrices (a construction explained, for instance, in \cite{brion2005lectures}[Section 1.2]). Let $T$ be the torus subgroup of $B$ consisting of diagonal matrices, then the left multiplication by $T$ on $G$ induces a $T$-action on $G/B$.
 In this setting, the Schubert varieties in $G/B$ are realised as the Zariski closures of the orbits in $G/B$ under the action of $B$, and they are invariant under the $T$-action. From this fact and from Theorem \ref{thm:isom}, we get an action of $T$ on the quiver Grassmannian $\Gr_{{\bf e}^w}(M)$ induced by the action of $T$ on $M$, which is in turn induced by the left multiplication of $T$ on the elements of the chosen basis $\mathcal{B}=\set{b_1, b_2,\dots, b_{n+1}}$ of $\C^{n+1}$ 
 (it is also straightforward to check that, if $N$ is an element of $\Gr_{{\bf e}^w}(M)$, then $T\cdot N$ is still in $\Gr_{{\bf e}^w}(M)$). Furthermore, following from its definition in \eqref{eq:mapsmooth}, the isomorphism $\psi: \Gr_{{\bf e}^w}(M) \to X_w$ is $T$-equivariant, i.e. $\varphi(t \cdot N)=t \cdot \varphi(N)$ for all $t\in T$ and $N \in \Gr_{{\bf e}^w}(M)$.
 
\end{remark}

\bibliography{bibliography.bib}
\bibliographystyle{alpha}
\end{document}